\documentclass[12pt]{amsart} 

\usepackage[margin=2.9cm]{geometry}             
\usepackage{color}
\usepackage{graphicx}
\usepackage{amssymb, enumerate,bm}
\usepackage[matrix,arrow,ps]{xy}
\usepackage{epstopdf, amsmath}
\usepackage{paralist}
\newcommand{\C}{\mathbb C}

\newcommand{\R}{\mathbb R}
\newcommand{\B}{\mathbb B}
\newcommand{\transp}{\,^t}

\newcommand{\aaa}{\mathfrak{(a)}}
\newcommand{\fff}{\mathfrak{(f)}}
\newcommand{\bbb}{\mathfrak{(b)}}
\newcommand{\ttt}{\mathfrak{(t)}}

\newcommand{\vnorm}[1]{{\| #1 \|}}

\DeclareMathOperator{\spanc}{span}
\newtheorem{theo}{Theorem}[section]
\newtheorem{lemma}[theo]{Lemma}
\newtheorem{cor}[theo]{Corollary}
\newtheorem{prop}[theo]{Proposition}

\theoremstyle{remark}
\newtheorem{remark}[theo]{Remark}
\theoremstyle{example}
\newtheorem{example}[theo]{Example}
\theoremstyle{definition}
\newtheorem{defi}[theo]{Definition}
\numberwithin{equation}{section}

\begin{document}

\begin{abstract} Let $M\subset \C^{N}$ be a generic real submanifold of class $\mathcal{C}^4$. 
In case $M$ is Levi non-degenerate in the sense Tumanov, we construct stationary discs for 
$M$. If furthermore $M$ satisfies an additional non-degeneracy condition, we apply the method of stationary discs to obtain $2$-jet determination of CR automorphisms of $M$.    

\end{abstract}

\author{Florian Bertrand, L\'ea Blanc-Centi and Francine Meylan}
\title[Stationary discs and finite jet determination]{Stationary discs and finite jet determination for non-degenerate generic real submanifolds}

\subjclass[2010]{}


\maketitle 


\section*{Introduction}
An important part  of the theory of functions of several complex variables going back to the pioneering work  of Poincar\'e \cite{Po},
is the understanding of biholomorphic equivalence of domains.
In dimension $N=1$, the 
Riemann mapping theorem
asserts that domains have only topological invariants.
A similar statement  fails in higher dimension, 
which is  connected with the fact
that smooth boundaries of domains in $\C^N$, $N\ge2$,
possess infinitely many local biholomorphic invariants.

The simplest 
test 
case is the real hyperquadric $Q$ given in 
$\C^N$ by the  equation
$$\Re e  w = \sum_{j=1}^{N-1} \pm |z_j|^2, \quad (z,w)\in \C^{N-1}\times \C.$$
From the infinite-dimensional family of all local biholomorphic self-maps of $\C^N$, only a finite-dimensional subfamily preserves $Q$.
That family, known as the  {\em (biholomorphic) automorphism group} of $Q$ plays
 a fundamental role in understanding automorphism groups of more general real hypersurfaces 
satisfying a non-degeneracy condition called {\em Levi non-degeneracy}. For a real submanifold $M\subset \C^N$ and a point $p \in M$, we denote by $Aut(M,p)$ the set of germs of biholomorphic maps  fixing $p$ and such that $F(M)\subset M$. 
We recall here the classical statement contained in the work of Chern and Moser \cite{ch-mo}.
\begin{theo}\cite{ch-mo}\label{emma}
If a  real hypersurface $M \subset \C^N $ is real analytic Levi non-degenerate at a point $p\in M,$ then elements of  $Aut(M,p)$ are uniquely determined by their  jets of order two at $p.$
\end{theo}

Let us remark that finite jet determination  for  holomorphic mappings between { \it real analytic} real  hypersurfaces has attracted  considerable attention these past years. For a survey, we refer for instance to the  articles of Zaitsev \cite{za1} or  Baouendi, Ebenfelt  and Rothschild \cite{BER2}. We also point out the works of Ebenfelt \cite{eb}, Ebenfelt and Lamel \cite{eb-la}, Kolar, Zaitsev and the third author \cite{KMZ1} in case of $\mathcal{C}^\infty$  real hypersurfaces. Finally the case of finitely smooth real hypersurfaces has been studied in \cite{be-bl, be-de1, be-de-la}.

The aim of this paper is to study  the finite jet determination problem for finitely smooth real submanifolds $M \subset  \C^N $  of codimension $d>1$ using the method developed by the first two authors in \cite{be-bl}. The present work is  motivated in particular  by the works of Beloshapka \cite{be2} and of the second  and third  authors \cite{bl-me2}. More precisely, according to Beloshapka \cite{be2}, if  $M\subset \C^N $ is real analytic  generic submanifold of codimension $d$ and Levi non-degenerate  at $p\in M$ in the sense of Beloshapka  then elements of  $Aut(M,p)$ are uniquely determined by their  jets of order two at $p.$
We emphasize  that various not equivalent  definitions  for the generalization to real submanifolds $M \subset \C^N $  of codimension $d>1$ of the notion  of  Levi non-degeneracy  condition for  a hypersurface   have  been introduced  in the literature  \cite {ba-ja-tr}, \cite{be2}, \cite{tu}. See also \cite {bl-me1} for a survey of these different notions.

Important work on finite jet determination for holomorphic mappings between real submanifolds of higher codimension has also 
been done; we refer  for instance to the articles of Zaitsev \cite{za}, Baouendi, Ebenfelt and Rotschild \cite{BER1}, Baouendi, Mir and Rotschild \cite{BMR}, Lamel and Mir 
\cite{la-mi}, Juhlin \cite{ju}, Juhlin and Lamel \cite{ju-la} for the real analytic case and Kim and Zaitsev \cite{ki-za} for the $
\mathcal{C}^\infty$ case.
We now state our  main result.
\begin{theo}\label{chloe}  
Let $M\subset \C^{N}$ be a $\mathcal{C}^4$ generic  real submanifold.    
Assume that $M$ is  fully non-degenerate at $p \in M.$  Then any germ  at $p$ of  CR automorphism   of $M$ of class $\mathcal{C}^3$   is  uniquely determined by its jet of order two at $p.$     
\end{theo}
We refer to Definition \ref{deffully} for the notion of  fully non-degenerate submanifold; this notion is closely related to 
the non-degeneracy condition introduced by Tumanov in \cite{tu}. Moreover, 
 in case $M\subset \C^{4}$ the fully non-degeneracy condition is equivalent to the non-degeneracy condition in the sense of Beloshapka, and thus we have
\begin{cor}\label{coro}
Let $M\subset \C^{4}$ be a $\mathcal{C}^4$ generic  real submanifold.    
 Assume that $M$ is  Levi non-degenerate  at $p \in M$  in the sense of Beloshapka. 
 Then any germ  at $p$ of  CR automorphism   of $M$ of class $\mathcal{C}^3$ is uniquely determined by its  jet of order two at $p.$  
\end{cor}
As a direct consequence of Theorem \ref{chloe}, we also have the following corollary.
\begin{cor}\label{zoe}
Let $M\subset \C^{N}$ be a $\mathcal{C}^4$ generic  real submanifold.   
 Assume that $M$ is  fully non-degenerate at $p \in M.$  Then elements of  $Aut(M,p)$  are uniquely determined by its  jet of order two at $p.$  
\end{cor}

Our approach is based on the study of an important family of invariant objects attached to real submanifolds, namely the stationary 
discs. These particular holomorphic discs were first introduced by Lempert in \cite{le} as  extremal discs for the Kobayashi 
metric of bounded smooth strongly convex domains in $\C^n$. 
The description of such discs and their applications were later on developed in more general settings such as strictly pseudoconvex hypersurfaces by Huang \cite{hu} and Pang \cite{pa} or in higher codimension by Tumanov \cite{tu}; see also the work of Sukhov and Tumanov \cite{su-tu} who construct stationary discs attached to small perturbations of $\mathbb{S}^3\times \mathbb{S}^3 \subset \C^4$, where $\mathbb{S}^3$ denotes the unit sphere in $\C^2$. Recently, the method of stationary discs has been particularly adapted to the study of jet determination problems for finitely smooth real hypersurfaces \cite{be-bl,be-de1,be-de-la}. 
In order to prove Theorem \ref{chloe}, we construct stationary discs attached to $M$; the idea of attaching such a disc to $M$ is a boundary value problem, namely a nonlinear Riemann-Hilbert type problem. Inspired by the work of Forstneri\v{c} \cite{fo} and of Globevnik \cite{gl1,gl2} on analytic  discs  attached  to  totally  real submanifolds, we analyze the existence and the structure of solutions of such nonlinear Riemann-Hilbert type problems (see Theorem \ref{theodiscs} and its Corollary \ref{cordiscs}). In order to study geometric properties of stationary discs, it is essential for the method developed in \cite{be-bl} to have the description of all  stationary discs attached to the model hyperquadric. Our approach, instead, is based on the choice of a "good" initial stationary disc and is therefore more subtle. Finally we emphasize that the fully non-degeneracy condition imposed on $M$ is natural in our context since it ensures that one  can construct  "enough" stationary discs with "good" geometric properties.

\

The paper is organized as follows. In  section $1,$  we introduce and discuss the notion of fully non-degeneracy condition  for a  generic  real submanifold  in the light of the various conditions for  non-degenerate real submanifolds introduced by Beloshapka and Tumanov. We also introduce the spaces of functions needed for the rest of the sequel and recall  the definition  and properties of stationary discs. In  section $2,$ we first construct an initial stationary disc for the the quadric part of the generic real submanifold that will enable us to construct  a  "big enough" family of stationary discs. In section $3,$ we discuss  crucial  geometric properties of this  family  that will allow us to prove  Theorem \ref{chloe} in section $4$.

\section{Preliminaries}

We denote by $\Delta$ the unit disc in $\C$, by $\partial \Delta$ its boundary,  and by $\B\subset \C^N$ the unit ball in $\C^N$.

\subsection{Non-degenerate generic real submanifolds} 
Let $M \subset \C^{N}$ be a  $\mathcal{C}^{2}$ generic   real submanifold of real codimension $d\ge 1$ through $p$. 
Under these hypotheses, and after a local biholomorphic change of coordinates, we may assume that $p=0$ and that $M\subset \C^N=\C^n_z\times\C^d_w$ is given locally by the following system of equations
\begin{equation}\label{eqred}
\begin{cases}
r_1= \Re e  w_1- \transp\bar z A_1 z+ O(3)=0\\
\ \ \ \ \vdots \\
r_d=\Re e  w_d- \transp\bar z A_d z+ O(3)=0
\end{cases}
\end{equation}
where $A_1,\hdots,A_d$ are Hermitian matrices of size $n$
(see \cite{ber} and Section 7.2  \cite{bo} for more details). In the remainder O(3), $z$ is of weight one and $\Im m w$ of weight two. We set $r=(r_1,\ldots,r_d)$.

We  recall the following definition introduced by Beloshapka in \cite{be1,be2}. 
\begin{defi} 
A $\mathcal{C}^{2}$ generic  real submanifold $M$ of $\C^N$ of codimension $d$ and given by (\ref{eqred}) is  \emph{Levi non-degenerate} at 0 in the sense of Beloshapka if the following two conditions are both satisfied
\begin{center}
\begin{tabular}{cl}
$\aaa$ & $A_1$,...,$A_d$ are linearly independent (equivalently on $\R$ or $\C$)\\
\\
$\bbb$ & $\bigcap_{j=1}^d\mathrm{Ker}A_j=\{0\}$
\end{tabular}
\end{center}
\end{defi}
The Levi non-degeneracy is a biholomorphic invariant notion (see for instance \cite{bl-me1}). Note that the non-degeneracy condition introduced in \cite{za,ber} is exactly condition $\bbb$. 
In our context, we need to work with a  stronger notion of non-degeneracy. 
\begin{defi}\label{deffully} 
A  $\mathcal{C}^{2}$ generic real submanifold $M$ of $\C^{N}$ of codimension $d$ given by Equations (\ref{eqred}) is \emph{fully non-degenerate} at 0 if the following two conditions are both satisfied    
\begin{center}
\begin{tabular}{cl}
$\fff$ & there exists $V \in \C^{n}$ such that ${\rm span}_{\C}\{A_1V,\ldots,A_dV\}$ is of dimension $d$,\\
\\
$\ttt$ & if there exists a real linear combination $\sum_{j=1}^d c_jA_j$ that is invertible.
\end{tabular}
\end{center}
\end{defi}
Note that $\fff$ implies that $d \leq n$ and also implies $\aaa$. We  point out that in case $d=2$  it can be checked that   $\fff$ and $\aaa$ are equivalent, but in general they are not equivalent as illustrated by the following example.
\begin{example} The quadric in $\C^8$ given by 
$$\begin{cases}
 \Re  e  w_1= |z_1|^2+|z_2|^2+|z_3|^2+|z_4|^2\\
 \Re  e  w_2= |z_1|^2\\
 \Re  e  w_3= |z_2|^2\\
 \Re  e  w_4= |z_1|^2+\Re e (\overline{z_1}z_2)\\
\end{cases}$$
is Levi non-degenerate at the origin but does not satisfy $\fff$. 
\end{example}  
Condition $\fff$ is tightly related to the existence of analytic discs whose centers fill an open set (see Proposition \ref{propcenteropen}) and determined by their $1$-jet (see Proposition \ref{propjetdiscs}).  Recall that condition $\ttt$ was introduced by Tumanov \cite{tu} and is essential for the construction of stationary disc (see Theorem \ref{theodiscs}); for this reason, we say that {\it $M$ is Levi non-degenerate at $0$ 
 in the sense of Tumanov}  if it satisfies conditon $\ttt$. Condition $\ttt$ implies $\bbb$ but not $\aaa$, except in the hypersurface case, since one can choose $A_1$ invertible and $A_2=\hdots=A_d=0$. Moreover, as pointed out in \cite{bl-me1}, the Levi non-degeneracy in the sense of Beloshapka does not necessarily imply $\ttt$. Finally we note that in case $M\subset \C^4$, the fully non-degeneracy condition and the Levi non-degeneracy in the sense of Beloshapka coincide.  

\subsection{Stationary discs}
We first introduce the spaces of functions we need. 
Let $k \geq 0$ be an integer and let $0< \alpha<1$.
We denote by $\mathcal C^{k,\alpha}=\mathcal C^{k,\alpha}(\partial\Delta,\R)$ the space of real-valued functions  defined on $\partial\Delta$ of class 
$\mathcal{C}^{k,\alpha}$. The space  $\mathcal C^{k,\alpha}$ is endowed with  its usual norm
$$\|f\|_{\mathcal{C}^{k,\alpha}}=\sum_{j=0}^{k}\|f^{(j)}\|_\infty+
\underset{\zeta\not=\eta\in \partial\Delta}{\mathrm{sup}}\frac{\|f^{(k)}(\zeta)-f^{(k)}(\eta)\|}{|\zeta-\eta|^\alpha},$$
where $\|f^{(j)}\|_\infty=\underset{\partial\Delta}{\mathrm{max}}\|f^{(j)}\|$.
We set $\mathcal C_\C^{k,\alpha} = \mathcal C^{k,\alpha} + i\mathcal C^{k,\alpha}$. The space $\mathcal C_\C^{k,\alpha}$ is equipped with the norm
$$\|f\|_{\mathcal{C}_{\C}^{k,\alpha}}=
\|\Re e  f\|_{\mathcal{C}^{k,\alpha}}+\|\Im m f\|_{\mathcal{C}^{k,\alpha}}$$ 
We denote by $\mathcal A^{k,\alpha}$ the subspace of {\it analytic discs} in $\mathcal C_{\C}^{k,\alpha}$ consisting of functions $f:\overline{\Delta}\rightarrow \C$, holomorphic on $\Delta$ with trace on 
$\partial\Delta$ belonging to $\mathcal C_\C^{k,\alpha}$.

 Let $M$ be a $\mathcal{C}^{2}$ generic  real submanifold of $\C^N$ of codimension $d$ given by  (\ref{eqred}). An analytic disc $f \in (\mathcal A^{k,\alpha})^{N}$ 
is {\it attached to  $M$} if $f(\partial \Delta) \subset M$. Following Lempert \cite{le} and Tumanov \cite{tu} we define
\begin{defi}
A holomorphic disc $f: \Delta \to \C^N$ continuous up to  $\partial \Delta$ and attached to  $M$ is {\it stationary for $M$} if there 
exists a  holomorphic lift $\bm{f}=(f,\tilde{f})$ of $f$ to the cotangent bundle $T^*\C^{N}$, continuous up to 
 $\partial \Delta$ and such that for all $\zeta \in \partial\Delta,\ \bm{f}(\zeta)\in\mathcal{N}M(\zeta)$
where
\begin{equation}\label{eqcon}
\mathcal{N}M(\zeta):=\{(z,w,\tilde{z},\tilde{w}) \in T^*\C^{N} \ | \ (z,w) \in M, (\tilde{z},\tilde{w}) \in 
\zeta N^*_{(z,w)} M\setminus \{0\} \},
\end{equation}
and where 
$$N^*_{(z,w)} M=\spanc_{\R}\{\partial r_1(z,w), \ldots, \partial r_d(z,w)\}$$ is the conormal fiber at $(z,w)$ of $M$. 
The set of these lifts $\bm{f}=(f,\tilde{f})$, with $f$ non-constant, is denoted by $\mathcal{S}(M)$.
\end{defi}
Note that equivalently, an analytic disc $f \in (\mathcal A^{k,\alpha})^{N}$ attached to  $M$ is stationary for $M$ if there exists $d$ real valued functions $c_1, \ldots, c_d : \partial \Delta \to \R$ such that $\sum_{j=1}^dc_j(\zeta)\partial r_j(0)\neq 0$ for all $\zeta \in \partial \Delta$   and such that the map 
$$\zeta \mapsto \zeta \sum_{j=1}^dc_j(\zeta)\partial r_j(f(\zeta), \overline{f(\zeta)})$$ defined on $\partial \Delta$ extends holomorphically on $\Delta$.

The set of such small discs is invariant under CR automorphisms; recall that if $F$ is a CR automorphism of $M$ and $f$ an analytic disc attached to $M$ then the map  $F\circ f$ defined on $\partial \Delta$ extends holomorphically to $\Delta$ (see Proposition 6.2.2 in \cite{ber} or Theorem 1 p. 200 in \cite{bo}). Moreover recall the following essential result due to Webster \cite{we} in the hypersurface case and to Tumanov \cite{tu} for higher codimension submanifolds.
\begin{prop}[\cite{tu}]\label{propco}
 Let $M$ be a $\mathcal{C}^{2}$ generic  real submanifold of $\C^N$ of codimension $d$ with local defining function of the form (\ref{eqred}). Then $M$ is Levi non-degenerate at $0$  in the sense of Tumanov if and only the conormal bundle $N^*M$ is totally real at $\left(0,\sum_{j=1}^dc_j\partial r_j(0)\right)$, where the $c_1,\ldots,c_d$ are such that $\sum_{j=1}^d c_jA_j$  is invertible.
 \end{prop}

We end this section with an important remark on the smoothness of stationary discs.  Let $M$ be a $\mathcal{C}^{4}$ generic  real submanifold of $\C^N$ of codimension $d$  with local defining function of the form (\ref{eqred}). Assume that $M$ is Levi non-degenerate at $0$ in the sense of Tumanov  $\ttt$. Consider a lift of stationary disc ${\bm f}=(f,\tilde{f})$ for $M$ satisfying ${\bm f}(1)=(0,\sum_{j=1}^d c_j\partial r_j(0))$ where $\sum_{j=1}^d c_jA_j$  is invertible. It follows from Proposition \ref{propco} and from Chirka (Theorem 33 in \cite{ch}) that such discs are of class $\mathcal{C}^{2,\alpha}$ for any $0<\alpha<1$ near $\zeta=1$.  

\subsection{Partial indices and Maslov index} The construction of stationary discs attached to a generic real submanifold of $\C^N$ (see Theorem \ref{theodiscs}) relies on a non-linear Riemann-Hilbert problem whose study is related to certain geometric integers, namely the partial indices and the Maslov index, associated to the linearized problem. In this section, we recall the definition of these integers. 
We denote by $Gl_N(\C)$ the general linear group on $\C^N$. 
Let $G: \partial\Delta \to Gl_N(\C)$ be a smooth map. 
We consider a Birkhoff factorization 
(see Section 3  \cite{gl1}  or  \cite{ve}) of $-\overline{G^{-1}}G$ on $\partial \Delta$:
$$ -\overline{G(\zeta)}^{-1}G(\zeta)=
B^+(\zeta)
\begin{pmatrix}
	\zeta^{\kappa_1}& & & (0) \\ &\zeta^{\kappa_2} & & \\ & & \ddots & \\ (0)& & &\zeta^{\kappa_{N}}
\end{pmatrix}
B^-(\zeta),$$
where  $\zeta \in \partial \Delta$, $B^+: \bar{\Delta}\to Gl_N(\C)$ and 
$B^-: (\C \cup \infty)\setminus\Delta\to Gl_N(\C)$  
 are  smooth  maps, holomorphic on $\Delta$ and $\C \setminus \overline{\Delta}$ respectively. 
The integers $\kappa_1, \dots, \kappa_N$  are called the {\it partial indices} of 
$-\overline{G^{-1}}G$ and  their sum 
$\kappa:=\sum_{j=1}^N\kappa_j$ the {\it Maslov index} of $-\overline{G^{-1}}G$.  Recall that  the Maslov index $\kappa$ is equal to 
the winding number  of the function 
$$\zeta \mapsto \det\left(-\overline{G(\zeta)^{-1}}G(\zeta)\right)$$ at the origin (\cite{gl2}, see also Lemma B.1 \cite{bl} for a proof);  here $\det\left(-\overline{G(\zeta)^{-1}}G(\zeta)\right)$ denotes the determinant of $-\overline{G(\zeta)^{-1}}G(\zeta)$.

\section{Construction of stationary discs}
In this section, we construct lifts of stationary discs attached to small pertubations of a quadric submanifold, non-degenerate in the sense of Tumanov $\ttt$. In particular we show that the set of such lifts form a finite dimensional submanifold of the Banach space of analytic discs.

\subsection{Construction of stationary discs in the quadric case}
 Consider a quadric submanifold $Q\subset \C^N=\C^n_z\times\C^d_w$ of real codimension $d$
\begin{equation}\label{eqquadric}
\begin{cases}
\rho_1= \Re  e w_1- \transp\bar z A_1 z = 0\\
\ \ \ \ \vdots \\
\rho_d=\Re e w_d - \transp\bar z A_d z = 0 
\end{cases}
\end{equation}
where $A_1,\ldots,A_d$ are hermitian matrices of size $n$. We set $\rho=(\rho_1,\ldots,\rho_d)$.  
We will determine a special family of lifts of stationary discs in $\mathcal{S}(Q)$.  
We suppose the quadric $Q$ to be Levi non-degenerate at $0$ in the sense of Tumanov $\ttt$. Note that the Levi non-degeneracy condition $\ttt$  is not strictly needed to determine some stationary discs for $Q$ but will be essential later in order to construct discs attached to small perturbations of $Q$ (see Theorem \ref{theodiscs}). We have 
\begin{equation*}
\begin{cases}
\displaystyle \partial \rho_1 = (\partial_z \rho_1, \partial_w \rho_1)=  \left(-\transp\bar z (A_1)_1,\ldots, -\transp\bar z (A_1)_n, \frac{1}{2},0\ldots,0\right)\\
\ \ \ \ \vdots \\
\displaystyle \partial \rho_d = (\partial_z \rho_d, \partial_w \rho_d)= \left(-\transp\bar z (A_d)_1,\ldots, -\transp\bar z (A_d)_n, 0\ldots,0,\frac{1}{2}\right)
 \end{cases}
\end{equation*}
where $(A_j)_l$ denotes, for $j=1,\dots,d$ and $l=1,\ldots,n$, the $l^{\rm th}$ column of $A_j$. Consider a disc 
$f=(h,g) \in (\mathcal A^{k,\alpha})^{n+d}$, with $h \in (\mathcal A^{k,\alpha})^{n}$ and 
$g \in (\mathcal A^{k,\alpha})^{d}$,  and $d$ real valued functions 
$c_1, \ldots, c_d : \partial \Delta \to \R$ such that $\sum_{j=1}^dc_j(\zeta)\partial \rho_j(0)\neq 0$. Note that if the map 
$\zeta \mapsto \zeta \sum_{j=1}^dc_j(\zeta)\partial \rho_j(f(\zeta), \overline{f(\zeta)})$ defined on $\partial \Delta$ extends holomorphically on 
$\Delta$ then we must have $c_j(\zeta)=a_j\overline{\zeta}+b_j+\overline{a_j}\zeta$ with $a_j \in \C$ and $b_j \in \R$. We restrict to the case where  $c=(c_1,\ldots,c_d)$ is a real constant vector-valued function and chosen such that $A:=\sum_{j=1}^d c_j A_j$ is invertible.  It follows that 
$$\zeta \sum_{j=1}^dc_j(\zeta)\partial \rho_j(f(\zeta), \overline{f(\zeta)})=\left(-\zeta\transp\overline{h(\zeta)} A,\frac{\zeta}{2}c\right).$$   
Consequently, $h$ must be   of the form
\begin{equation}\label{eqdisch}
h(\zeta)=V+\zeta W,
\end{equation}
with $V \in \C^n$,  $W \in \C^n\setminus\{0\}$ and $g$ is thus of the form   
\begin{equation}\label{eqdiscg}
g_j(\zeta)=\transp\overline{V}A_jV+\transp\overline{W}A_jW+2\transp\overline{V}A_jW\,\zeta+iy_j 
\end{equation}
with $y_j \in \R$, $j=1,\ldots,d$.  Note that, in such case, we have 
\begin{equation}\label{eqdisctildeh}
\tilde{h}(\zeta)=-\zeta\transp\overline{h(\zeta)} A
\end{equation}
and 
\begin{equation}\label{eqdisctildeg}
\tilde{g}(\zeta)=\frac{\zeta}{2}c.
\end{equation}

\subsection{Defining equations of the conormal bundle and Riemann-Hilbert problem}
Let  $Q\subset \C^{n+d}$ be a quadric submanifold  of real codimension $d$ defined by $\{\rho=0\}$ (\ref{eqquadric}). Denote by $(z,w,\tilde{z},\tilde{w})$ the coordinates on $T^*\C^{n+d}$, with $\tilde{z}=(\tilde{z_1},\ldots,\tilde{z_n})$ and $\tilde{w}=(\tilde{w_1},\ldots,\tilde{w_d})$.  The $2n+2d$ real defining equations for $\mathcal{N}Q(\zeta),$ $\zeta \in \partial\Delta$, (see  (\ref{eqcon}) for its definition) are given by 

\begin{equation*}
\left\{
\begin{array}{lll} 

\displaystyle \tilde{\rho}_1(\zeta)(z,w,\tilde{z},\tilde{w}) & = & \displaystyle \frac{w_1+\overline{w_1}}{2} - \transp\bar{z} A_1z  = 0,\\
\vdots& \vdots&\vdots\\
\tilde{\rho}_d(\zeta)(z,w,\tilde{z},\tilde{w}) & = & \displaystyle \frac{w_d+\overline{w_d}}{2} - \transp\bar{z} A_dz  = 0,\\
\tilde{\rho}_{d+1}(\zeta)(z,w,\tilde{z},\tilde{w}) & = & \left(\tilde{z_1}+2 \transp\bar{z} \sum_{j=1}^d \tilde{w_j}(A_j)_1\right)+\left(\overline{\tilde{z_1}}+2 \transp z \sum_{j=1}^d \overline{\tilde{w_j}(A_j)_1}\right)= 0,\\
\vdots& \vdots&\vdots\\
\tilde{\rho}_{d+n}(\zeta)(z,w,\tilde{z},\tilde{w}) & = & \left(\tilde{z_n}+2 \transp\bar{z} \sum_{j=1}^d \tilde{w_j}(A_j)_n\right)+\left(\overline{\tilde{z_n}}+2 \transp z \sum_{j=1}^d \overline{\tilde{w_j}(A_j)_n}\right)=0,\\
\tilde{\rho}_{d+n+1}(\zeta)(z,w,\tilde{z},\tilde{w}) & = &  i\left(\tilde{z_1}+2 \transp\bar{z} \sum_{j=1}^d \tilde{w_j}(A_j)_1\right)-i\left(\overline{\tilde{z_1}}+2 \transp z \sum_{j=1}^d \overline{\tilde{w_j}(A_j)_1}\right)= 0,\\
\vdots& \vdots&\vdots\\
\tilde{\rho}_{2n+d}(\zeta)(z,w,\tilde{z},\tilde{w}) & = &  i\left(\tilde{z_n}+2 \transp\bar{z} \sum_{j=1}^d \tilde{w_j}(A_j)_n\right)-i\left(\overline{\tilde{z_n}}+2 \transp z \sum_{j=1}^d \overline{\tilde{w_j}(A_j)_n}\right) = 0,\\
\tilde{\rho}_{2n+d+1}(\zeta)(z,w,\tilde{z},\tilde{w}) & = & \displaystyle i\frac{\tilde{w_1}}{\zeta}-i\zeta\overline{\tilde{w_1}} = 0,\\
\vdots& \vdots&\vdots\\
\tilde{\rho}_{2n+2d}(\zeta)(z,w,\tilde{z},\tilde{w}) & = & \displaystyle i\frac{\tilde{w_d}}{\zeta}-i\zeta\overline{\tilde{w_d}} = 0,\\
\end{array}
\right.
\end{equation*}
where $(A_j)_l$ denotes, for $j=1,\dots,d$ and $l=1,\ldots,n$, the $l^{\rm th}$ column of $A_j$.  We set $\tilde{\rho}:=(\tilde{\rho_1},\ldots,\tilde{\rho}_{2n+2d})$. For a general quadric $M=\{r=0\}$ with $r=(r_1,\ldots,r_d)$ of the form ({\ref{eqred}}), 
 we denote by 
$\tilde{r}$ the corresponding defining functions of $\mathcal{N}M(\zeta)$. This allows to consider lifts of stationary discs as solutions of a nonlinear 
Riemann-Hilbert type problem. Indeed, an analytic disc $\bm{f} \in \left(\mathcal{A}^{k,\alpha}\right)^{2n+2d}$ is the lift of a stationary disc for $M=\{r=0\}$ 
if and only if 
\begin{equation}\label{eqstatrh}
\tilde{r}(\bm{f})=0 \ \mbox{ on } \partial\Delta.
\end{equation}
The next section is devoted to the  study of this problem. 
 \subsection{Construction of stationary discs}\label{secconst}
We recall that the set of lifts of non-constant stationary discs  attached to a generic real submanifold $M\subset \C^N$ is  denoted by $\mathcal{S}(M)$. We have the following theorem which is crucial in our approach of the  jet determination problem. 
\begin{theo}\label{theodiscs}

Let  $Q\subset \C^{n+d}$ be a quadric submanifold  of real codimension $d$ defined by $\{\rho=0\}$ (\ref{eqquadric}). 
 Assume that $Q$ is Levi non-degenerate at $0$ in the sense of Tumanov $\ttt$.
Consider an initial lift of a stationary disc, $\bm{f_0}=(h_0,g_0,\tilde{h_0},\tilde{g_0})$ of the form  (\ref{eqdisch}), 
 (\ref{eqdiscg}),  (\ref{eqdisctildeh}),  (\ref{eqdisctildeg}) where $c_1,\ldots,c_d$ are chosen such that the matrix 
 $A:=\sum_{j=1}^d c_jA_j$ is invertible. Then there exist  open 
neighborhoods $U$ of $\rho$ in $(\mathcal{C}^4(\B))^{d}$  and $V$ of $0$  in $\R^{4n+4d}$, a real number $\varepsilon>0$ and a map
$$\mathcal{F}:U \times V \to  \left(\mathcal{A}^{1,\alpha}\right)^{2n+2d} $$
 of class $\mathcal{C}^1$ such that:
\begin{enumerate}[i.]
\item $\mathcal{F}(\rho,0)=\bm{f_0}$,
\item for all $r\in U$, the map 
$$\mathcal{F}(r,\cdot):V\to \{\bm{f} \in \mathcal{S}(\{r=0\})\ \ | \  
\|\bm{f}-\bm{f_0}\|_{1,\alpha}<\varepsilon\}$$
is one-to-one and onto.
\end{enumerate}
\end{theo}
As a direct corollary, we obtain
\begin{cor}\label{cordiscs}
Under the assumptions of Theorem \ref{theodiscs}, there exist  open 
neighborhoods $U$ of $\rho$ in $(\mathcal{C}^4(\B))^{d}$ and a real number $\varepsilon>0$ such that 
$$\{\bm{f} \in \mathcal{S}(\{r=0\})\ \ | \  
\|\bm{f}-\bm{f_0}\|_{1,\alpha}<\varepsilon\}$$ 
forms a $\mathcal{C}^1$ real submanifold of dimension $4n+4d$ of the Banach space of analytic discs.    
\end{cor}

\begin{remark} Working with the Banach spaces $\mathcal{C}^4(\B)$ and $\mathcal{A}^{1,\alpha}$ is crucial for our approach, which is based on the implicit function theorem. The required smoothness is indeed necessary  for the below map $F$ (see (\ref{eqF})) to be $\mathcal{C}^1$.     
\end{remark}
\begin{remark}
In \cite{su-tu}, Sukhov and Tumanov prove Theorem \ref{theodiscs} in the special case the model quadric is $\mathbb{S}^3\times \mathbb{S}^3 \subset \C^4$, where $\mathbb{S}^3$ denotes the unit sphere in $\C^2$ (see Corollary 3.2 and Theorem 3.4 \cite{su-tu}). Their approach also relies on the study of the corresponding Riemann-Hilbert problem using the methods developed by \cite{fo,gl1,gl2}. See also \cite{tu, sc-tu}. 
\end{remark}
 \begin{proof}[Proof of Theorem \ref{theodiscs}]
In a neighborhood of $(\rho,\bm{f_0})$ in 
$(\mathcal{C}^4(\B))^{d} \times \left(\mathcal{A}^{1,\alpha}\right)^{2n+2d}$,
 we define the following map between Banach spaces
$$F: (\mathcal{C}^4(\B))^{d} \times \left(\mathcal{A}^{1,\alpha}\right)^{2n+2d}
\to (\mathcal C^{1,\alpha})^{2n+2d} $$
by 
\begin{equation}\label{eqF}
F(r,\bm{f}):=\tilde{r}(\bm{f}).
\end{equation}
Here we use the notation
$$\tilde{r}(\bm{f})(\zeta)=\tilde{r}(\zeta)(\bm{f}(\zeta)), \ \zeta \in \partial \Delta.$$
The map $F$ is of class $\mathcal{C}^1$ (see Lemma 5.1 in \cite{hi-ta} and Lemma 6.1 and Lemma 11.2 in \cite{gl1} which generalizes to $\mathcal{C}^{1,\alpha}$). Recall that an analytic disc $\bm{f} \in \left(\mathcal{A}^{1,\alpha}\right)^{2n+2d}$ is the lift of a stationary disc for  $\{r=0\}$ if and only if it solves the nonlinear Riemann-Hilbert problem (\ref{eqstatrh}). In other words, 
 for any fixed $r\in (\mathcal{C}^4(\B))^{d}$, 
 the zero set of $F(r,\cdot)$ coincides with 
$\mathcal{S}(\{r=0\})$. In order to prove Theorem \ref{theodiscs}, we apply the implicit function theorem to the map $F$. We need to consider the partial derivative of $F$ with respect to $\left(\mathcal{A}^{1,\alpha}\right)^{2n+2d}$ at 
$(\rho,\bm{f_0})$
\begin{equation}\label{eqrh0}
\partial_2 F(\rho,\bm{f_0})\bm{f}=2\Re e  \left[\overline{G(\zeta)}\bm{f}\right]
\end{equation}
where $G(\zeta)$ is the following complex valued square matrix of size $2n+2d$  
$$G(\zeta):=\left({\tilde\rho}_{\overline{z}}(\bm{f_0}),
{\tilde\rho}_{\overline{w}}(\bm{f_0}),
{\tilde\rho}_{\overline{\tilde{z}}}(\bm{f_0}),{\tilde\rho}_{\overline{\tilde{w}}}(\bm{f_0})\right).$$
Recall that the non-degeneracy in the sense of Tumanov is equivalent to the fact that the conormal bundle is totally real (see Proposition 
\ref{propco}); with respect to the choice of the initial disc, this ensures that the matrix $G(\zeta)$ in invertible for all $\zeta \in \partial \Delta$. For  sake of clarity, this fact will be 
proved again below.  
We need to show that (see p. 39 \cite{gl2})
\begin{enumerate}[i.]
\item the map $\partial_2 F(\rho,\bm{f_0})$ is onto, and
\item the real dimension of the kernel of $\partial_2 F(\rho,\bm{f_0})$ is $4n+4d$. 
\end{enumerate}
\subsubsection{Surjectivity of  $\partial_2 F(\rho,\bm{f_0})$}
It is more convenient to reorder coordinates and consider $(w,z,\tilde{z},\tilde{w})$ instead of 
$(z,w,\tilde{z}, \tilde{w})$. Accordingly, discs $\bm{f}$ are of the form $(g,h,\tilde{h},\tilde{g})$. We still denote by $G(\zeta)$ the corresponding reordered matrix, namely
    $$G(\zeta):=\left({\tilde\rho}_{\overline{w}}(\bm{f_0}),
{\tilde\rho}_{\overline{z}}(\bm{f_0}),
{\tilde\rho}_{\overline{\tilde{z}}}(\bm{f_0}),{\tilde\rho}_{\overline{\tilde{w}}}(\bm{f_0})\right).$$
The matrix $G(\zeta)$ is square of size $2n+2d$,  upper block triangular and given by 
\begin{equation}\label{eqG}
G(\zeta)=\left(\begin{array}{cccccccccccc}
\frac{1}{2}I_d &  & (*)\\
   & G_2(\zeta) &  \\
 (0) &   &  -i\zeta I_d\\
\end{array}\right),
\end{equation}
where $I_d$ denotes the identity matrix of size $d$ and $G_2(\zeta)$ is the following square matrix of size $2n$
\begin{equation}\label{eqG_2first}
G_2=\left(\begin{matrix}
2\sum_{j=1}^d{\tilde{g_j}}(A_j)_{11}& \hdots  &2\sum_{j=1}^d{\tilde{g_j}}(A_j)_{n1}     & 1  &    &  0 \\
\vdots & \vdots & \vdots &  & \ddots  \\

2\sum_{j=1}^d{\tilde{g_j}}(A_j)_{1n}&  \hdots &    2\sum_{j=1}^d{\tilde{g_j}}(A_j)_{nn} &0   &     & 1\\

2i\sum_{j=1}^d{\tilde{g_j}}(A_j)_{11}  &  \hdots & 2i\sum_{j=1}^d{\tilde{g_j}}(A_j)_{n1}  &  -i  &     &  0 \\
\vdots & \vdots & \vdots &  & \ddots  \\

2i\sum_{j=1}^d{\tilde{g_j}}(A_j)_{1n}  &   \hdots   & 2\sum_{j=1}^d{\tilde{g_j}}(A_j)_{nn}   &  0   &  & -i \\
\end{matrix}\right).
\end{equation}
where $\tilde{g_j}=c_j/2\zeta$ and $(A_j)_{kl}$, $k=1,\ldots,n$, $l=1,\ldots,n$,  denotes the $kl$ coefficient of $A_j$. Recall   that the
real constants $c_1,\ldots,c_d$ are chosen such that $A:=\sum_{j=1}^d c_j A_j$ is invertible. It follows that  
$$G_2(\zeta)=\left(\begin{matrix}
\zeta \transp A& I_n &     \\

i\zeta \transp A & -iI_n       \\

\end{matrix}\right)$$
is invertible on $\partial \Delta$, and that, accordingly, so is $G(\zeta)$. Due to the expression of $G(\zeta)$, in order to show its sujectivity, it is enough to show that the map 
$$L_1: \left(\mathcal{A}^{1,\alpha}\right)^{2n} \to (\mathcal C^{1,\alpha})^{2n}$$
defined by $L_1=2\Re e  \left[\overline{G_2(\zeta)} \ \cdot \right]$ is surjective.  For this purpose, we will show that the partial indices $k_1,\ldots,k_{2n}$ of  
$-\overline{G_2^{-1}}G_2$ are nonnegative (see \cite{gl1} or Section 4 in \cite{gl2}). Right multiplication by the constant matrix 
$\left(\begin{array}{cc} \transp A^{-1}&0\\ 0&I_{n}  \end{array}\right)$ does not change the partial indices, and gives us the matrix 
 $$\left(\begin{matrix}
\zeta I_n& I_n &     \\

i\zeta I_n& -iI_n       \\

\end{matrix}\right)$$
After permuting rows and columns, which also does not change the partial indices, we obtain 
\begin{equation}\label{eqG_2}
G_2^{\flat }=\left(\begin{matrix}
R&  & 0      \\
 & \ddots &   \\
0  & & R        \\
\end{matrix}\right), \mbox{ with } R(\zeta)=\left(\begin{matrix}
\zeta & 1      \\
i\zeta & -i   \\
\end{matrix}\right).
\end{equation}
By a direct computation we have
$$-\overline{(G_2^{\flat})^{-1}}{G_2^{\flat}}=\left(\begin{matrix}
P&  & 0      \\
 & \ddots &   \\

0  & & P         \\
\end{matrix}\right) \mbox{ with } P(\zeta)=-\left(\begin{matrix}
0 & \zeta       \\
\zeta & 0  \\
\end{matrix}\right),$$
which, for instance, decomposes as
which decomposes as
\begin{equation}\label{eqP}
P(\zeta)=\left(\begin{matrix}
0 & -1       \\
-1 & 0  \\
\end{matrix}\right) \left(\begin{matrix}
\zeta &  0      \\
0 & \zeta  \\
\end{matrix}\right) \left(\begin{matrix}
1 & 0   \\
0 & 1  \\
\end{matrix}\right).
\end{equation}
It follows that the partial indices of $-\overline{(G_2^{\flat})^{-1}}{G_2^{\flat}}$ are all equal to one and that, therefore, the map $\partial_2 F(\rho,\bm{f_0})$
is onto.

\subsubsection{Kernel of  $\partial_2 F(\rho,\bm{f_0})$}
Recall (see \cite{gl1} or Section 5 in \cite{gl2}) 
 that the real dimension  of the kernel of  $\partial_2 F(\rho,\bm{f_0})$ is given by $\kappa+2n+2d,$
where $\kappa$ is the Maslov index of $-\overline{G(\zeta)^{-1}}G(\zeta)$ and is equal to  
the winding number  of the function $\zeta \mapsto \det\left(-\overline{G(\zeta)^{-1}}G(\zeta)\right)$ at the origin.
A direct computation - using the form of $G(\zeta)$ and in particular $G_2(\zeta)$ - shows that $\kappa=2n+2d$.
 
\end{proof}
\section{Geometric properties of stationary discs}
In \cite{be-bl}, the first two authors describe geometric properties of stationary discs attached to small deformations  of a given Levi non-degenerate hyperquadric; it is important for that approach to have an explicit description of all  stationary discs attached to the model hypersurface. In our case, however, it does not seem possible to obtain an explicit description of  all  stationary discs attached to the model quadric in general, and therefore, it is more subtle to obtain that centers of stationary discs fill an open set (Proposition \ref{propcenteropen}) and that their lifts are determined by their $1$-jet at $1$ (Proposition \ref{propjetdiscs}).

\subsection{Discs with pointwise constraints} 
In  Section \ref{secfill}, we wish to show that centers of stationary discs fill an open set. This will be achieved by showing that the map 
${\bm f} \mapsto f(0) \in \C^{n+d}$ restricted to a smaller family of discs is a diffeomorphism onto its image. For dimensional reasons, it is essential to 
impose pointwise constraints on stationary discs in order to obtain a $2n+2d$ real dimensional family (see Theorem \ref{theodiscscons}).

We consider a quadric submanifold $Q\subset \C^{n+d}$ of real codimension $d$ given by $\{\rho=0\}$ (\ref{eqquadric}), Levi non-degenerate at $0$ in 
the sense of Tumanov $\ttt$ and an initial  lift of a stationary disc $\bm{f_0}=(h_0,g_0,\tilde{h_0},\tilde{g_0})$ of the form  (\ref{eqdisch}), (\ref{eqdiscg}),  
(\ref{eqdisctildeh}),  (\ref{eqdisctildeg}) where $c_1,\ldots,c_d$ are chosen such that the matrix  $A:=\sum_{j=1}^d c_jA_j$ is invertible. 
We wish to show that for $r$ close enough to $\rho$ and for  some positive $\varepsilon$, the set  
$$\left\{{\bm f} \in \mathcal{S}(\{r=0\}) , \ 
\|\bm{f}-\bm{f_0}\|_{1,\alpha}<\varepsilon, \ f(1)=0, \ \tilde{g}(1)=\left(\frac{c_1}{2},\ldots,\frac{c_d}{2}\right)\right\}$$ is a real submanifold of dimension $2n+2d$ of the Banach space of analytic discs. 
To this end we  introduce spaces of analytic discs with prescribed pointwise constraints.  
We denote by $\mathcal A^{1,\alpha}_{0}$ the subspace of $\mathcal C_{\C}^{1,\alpha}$ of functions of the form $(1-\zeta) f$, with  
$f\in \mathcal A^{1,\alpha}$.  We equip 
$\mathcal A^{1,\alpha}_{0}$ with the following norm 
\begin{equation}\label{eqnorm}
\|(1-\zeta) f\|_{\mathcal A^{1,\alpha}_{0}}
=\vnorm{ f }_{\mathcal{C}_{\C}^{1,\alpha}}
\end{equation}
which makes it  a Banach space and isomorphic to $\mathcal A^{1,\alpha}$.  We also denote by 
$\mathcal C_{0}^{1,\alpha}$ the subspace of $\mathcal C^{1,\alpha}$ of functions of the form  $(1-\zeta) v$
 with $v\in \mathcal C_\C^{1,\alpha}$. The space $\mathcal C_{0}^{1,\alpha}$ is equipped with the norm
$$\|(1-\zeta) f\|_{\mathcal C_{0}^{1,\alpha}}=\vnorm{ f }_{\mathcal C_\C^{1,\alpha}}.$$ 
Note that $\mathcal C_{0}^{k,\alpha}$ is a Banach space. 

Recall that the initial lift of stationary disc  $\bm{f_0}=(h_0,g_0,\tilde{h_0},\tilde{g_0})$  that  we are considering is of the form  (\ref{eqdisch}), 
 (\ref{eqdiscg}),  (\ref{eqdisctildeh}),  (\ref{eqdisctildeg}) where $c=(c_1,\ldots,c_d)$ is chosen such that the matrix 
 $\sum_{j=1}^d c_jA_j$ is invertible. Assume furthermore  that $W=-V$; more explicitly 
 \begin{equation}\label{eqinit}
 \bm{f_0}=\left((1-\zeta)V,2(1-\zeta)\transp \overline{V}A_1V,\ldots,2(1-\zeta)\transp \overline{V}A_dV,(1-\zeta)\transp \overline{V}A, \frac{\zeta}{2} c\right).
 \end{equation} 
We define the affine space $\mathcal{A}$ to be the subset of $(\mathcal{A}^{1,\alpha})^{2n+2d}$ of discs of the form 
$$\left((1-\zeta)h,(1-\zeta)g,(1-\zeta)\tilde{h},(1-\zeta)\tilde{g}+\frac{\zeta}{2} c\right),$$ 
 and we set
$$\mathcal{S}_0(\{r=0\})=\mathcal{S}(\{r=0\})\cap \mathcal{A}.$$ 
Notice that $\bm{f_0} \in \mathcal{S}_0(\{r=0\})$. We have
  \begin{theo}\label{theodiscscons}
Let $Q\subset \C^{n+d}$ be a quadric submanifold of real codimension $d$ given by (\ref{eqquadric}), Levi non-degenerate at $0$ in the sense of Tumanov $\ttt$.  Let $\bm{f_0}=(h_0,g_0,\tilde{h_0},\tilde{g_0})$ a lift of a stationary disc for $Q$  of the form  (\ref{eqinit}). Then there exist  open 
neighborhoods $U$ of $\rho$ in $(\mathcal{C}^4(\B))^{d}$  and $V$ of $0$  in $\R^{2n+2d}$, a real number $\varepsilon>0$ and a map
$$\mathcal{F}_0:U \times V \to  \mathcal{A} $$
 of class $\mathcal{C}^1$ such that:
\begin{enumerate}[i.]
\item $\mathcal{F}_0(\rho,0)=\bm{f_0}$,

\item for all $r\in U$ the map 
$$\mathcal{F}_0(r,\cdot):V\to \{\bm{f} \in \mathcal{S}_0(\{r=0\})\ \ | \  
\|\bm{f}-\bm{f_0}\|_{\mathcal A^{1,\alpha}_{0}}<\varepsilon\}$$
is one-to-one and onto.
\end{enumerate}
\end{theo}
We will need the following theorem from \cite{be-de2}.
 \begin{theo}[Theorem 2.4 \cite{be-de2}]\label{theoind}
Let $G: \partial \Delta \to  GL_N(\C)$ be a smooth map of the form
$$
G(\zeta)=\left(\begin{array}{cccc}G_1(\zeta)& & & (*) \\ &G_2(\zeta) & & \\ & & \ddots & \\ (0)& & &G_r(\zeta)\end{array}\right),
$$
where $G_j(\zeta) \in GL_{N_j}(\C)$ for all $j=1,\ldots,r$, for all $\zeta \in \partial \Delta$. Let $0<\alpha<1$. Consider the following operator 
$$L: \left(\mathcal{A}^{1,\alpha}_{0}\right)^{N}
\to  \left(\mathcal C_{0}^{1,\alpha}\right)^{N}$$
defined by
\begin{equation*}
L({\bm f})=2\Re e \left[\overline{G}{\bm f}\right].
\end{equation*}
 For  $j=1,\ldots,r$ we denote by $\kappa_{l}^{j}$, $l=1,\ldots, N_j$, the partial indices of  $-\overline{G_j}^{-1}G_j$ and by $\kappa$ the 
Maslov index of $-\overline{G}^{-1}G$.
We have
\begin{enumerate}[(i)]
\item If $\kappa_{l}^j \geq 0$ for all \ $l=1,\ldots,N_j$  
and  $j=1,\ldots,r$ then the map $L$ is onto.  
\item Assume that $L$ is onto. Then the kernel of $L$ has real dimension $\kappa$. 
\end{enumerate}
\end{theo}

\begin{proof}[Proof of Theorem \ref{theodiscscons}]
The proof is basically a repetition of the one of Theorem \ref{theodiscs}. 
We consider the corresponding map between Banach spaces
$$F_0: (\mathcal{C}^4(\B))^{d} \times (\mathcal{A}^{1,\alpha}_{0})^{2n+2d} 
\to (\mathcal C^{1,\alpha}_0)^{2n+2d}$$
defined by 
\begin{equation}\label{eqF_0}
F_0(r,\bm{f}):=\tilde{r}\left(h,g,\tilde{h},\tilde{g}+\frac{\zeta}{2}c\right), 
\end{equation}
where $\bm{f}=(h,g,\tilde{h},\tilde{g}).$ 
Its differential $\partial_2 F_0(\rho,\bm{f_0}^*)$ is still of the form (\ref{eqrh0}) where  $G$ is given by (\ref{eqG}); here $\bm{f_0}^*$ denotes the disc
 \begin{equation}\label{discini}
 \bm{f_0}^*=\left((1-\zeta)V,2(1-\zeta)\transp \overline{V}A_1V,\ldots,2(1-\zeta)\transp \overline{V}A_dV,(1-\zeta)\transp \overline{V}A, 0\right).
 \end{equation}  The surjectivity and real dimension of its kernel are now insured by Theorem \ref{theoind}; more precisely, strictly following the notations of Theorem \ref{theoind}, we have $r=3$,  
$G_1(\zeta)=\frac{1}{2}I_d$ with $d$ indices all equal to zero,  $G_2(\zeta)$ given by (\ref{eqG_2first}) with $2n$ indices all equal to $1$ and $G_3(\zeta)=-i\zeta I_d$ with $d$ indices all equal to two. Therefore the 
 differential $\partial_2 F_0(\rho,\bm{f_0}^*)$ is surjective and the real dimension of its kernel is $\kappa=2n+2d$. 
 \end{proof}

\subsection{Filling an open set}\label{secfill} 
We now prove that in case the quadric $Q$ is fully non-degenerate at $0$, the family of stationary discs fills an open neighborhood of the origin. We point out that the fact that $Q$ satisfies condition $\fff$, namely that there exists $V\in \C^n$ such that ${\rm span}_{\C}\{A_1V, \ldots, A_dV\}$ is of dimension $d$, is essential here.
\begin{prop}\label{propcenteropen}
Let $Q\subset \C^{n+d}$ be a quadric submanifold of real codimension $d$ given by (\ref{eqquadric}), fully non-degenerate at $0$.  
Consider an initial disc $\bm{f_0}$ of the form 
 $$\bm{f_0}=\left((1-\zeta)V,2(1-\zeta)\transp \overline{V}A_1V,\ldots,2(1-\zeta)\transp \overline{V}A_dV,(1-\zeta)\transp \overline{V}A,\frac{\zeta}{2}c\right)$$ 
where $V$ is given by $\fff$ and $c=(c_1,\ldots,c_d)$ is chosen such that the matrix 
 $\sum_{j=1}^d c_jA_j$ is invertible. Then there exist  an open 
neighborhood $U$ of $\rho$ in $(\mathcal{C}^4(\B))^{d}$ 
and a positive $\varepsilon$ such that for all $r \in U$ the set  $$\{f(0) \ \ | \ {\bm f} \in \mathcal{S}_0(\{r=0\}) , \ 
\|\bm{f}-\bm{f_0}\|_{\mathcal A^{1,\alpha}_{0}}<\varepsilon \}$$ contains an open set of $\C^{n+d}$. 
\end{prop}  We will prove that there exist an open 
neighborhood $U$ of $\rho$ in $(\mathcal{C}^4(\B))^{d}$ 
and a positive $\varepsilon$ such that for all $r \in U$ the map 
$$\Psi_r:  \{{\bm f} \in \mathcal{S}_0(\{r=0\}) , \ \|\bm{f}-\bm{f_0}\|_{\mathcal A^{1,\alpha}_{0}}<\varepsilon\} \to \C^{n+d}$$ defined by
$$\Psi_r({\bm f})=f(0)$$ 
is a diffeomorphism onto its image. We consider the corresponding differential
$$d_{\bm {f_0}} \Psi_{\rho}: T_{\bm {f_0}}\mathcal{S}_0(Q) \to \C^{n+d}$$ 
with $d_{\bm {f_0}} \Psi_{\rho}({\bm f})=f(0)$. Recall that according to  Theorem \ref{theodiscscons}, the tangent space 
$T_{\bm {f_0}}\mathcal S_0(Q)$ is the kernel of the differential $\partial_2 F_0(\rho,\bm{f_0}^*)= 2\Re e \left[\overline{G(\zeta)}\ \cdot \right]$ where $F_0$, $\bm{f_0}^*$ and 
$G$ are respectively defined in (\ref{eqF_0}), (\ref{discini}) and (\ref{eqG}), and is of dimension $2n+2d$. Proposition \ref{propcenteropen} is then a consequence of the following lemma. 
\begin{lemma}\label{leminj}
The linear map $d_{\bm {f_0}} \Psi_{\rho}: T_{\bm {f_0}}\mathcal{S}_0(Q) \to \C^{n+d}$ is injective.
\end{lemma} 
\begin{proof}We follow strictly the notations introduced in the proof of Theorem \ref{theodiscs}. 
Let $\bm{f}=(1-\zeta)(h,g,\tilde{h},\tilde{g})$ be an element of $T_{\bm {f_0}}\mathcal{S}_0(Q)$ in the kernel of $d_{\bm {f_0}} \Psi_{\rho}$, that is, satisfying $f(0)=0$. Since the tangent space 
$T_{\bm {f_0}}\mathcal S_0(Q)$ is the kernel of the differential $\partial_2 F_0(\rho,\bm{f_0}^*)= 2\Re e \left[\overline{G(\zeta)}\ \cdot \right]$, we have
\begin{equation}\label{eqker2}
\overline{G(\zeta)}\bm{f}+G(\zeta)\overline{\bm{f}}=0,
\end{equation}
where $G$ is of the form (\ref{eqG}); more explicitly $G$ is of the form 
\begin{equation}\label{eqGprecise}
G(\zeta)=\left(\begin{array}{cccccccccccc}
\frac{1}{2}I_d & B(\zeta)& 0\\
 0& G_2(\zeta) &  C(\zeta)\\
 0 & 0 &  -i\zeta I_d\\
\end{array}\right), \zeta \in \partial \Delta,
\end{equation}
where $I_d$ denotes the identity matrix of size $d$, $B(\zeta)$ is the following  matrix of size $d\times 2n$
\begin{equation}\label{eqB}
B=\left(\begin{matrix}
-(A_1)^1 h_0(\zeta)& \hdots  &-(A_1)^n h_0(\zeta) & 0  & \hdots     &  0 \\
\vdots &  & \vdots &  \vdots & & \vdots \\

-(A_d)^1 h_0(\zeta)& \hdots  &-(A_d)^n h_0(\zeta) & 0  & \hdots     &  0 \\

\end{matrix}\right)=(1-\zeta)B_1,
\end{equation}
and $C(\zeta)$ is the following  matrix of size $2n \times d$ 
$$C=\left(\begin{matrix}
2 \transp h_0 \overline{(A_1)_1} & \hdots  & 2 \transp h_0 \overline{(A_d)_1} \\
\vdots &  & \vdots &    \\

2 \transp h_0 \overline{(A_1)_n} & \hdots  & 2 \transp h_0 \overline{(A_d)_n} \\
-2i \transp h_0 \overline{(A_1)_1} & \hdots  & -2i \transp h_0 \overline{(A_d)_1} \\
\vdots &  & \vdots &    \\

-2i \transp h_0 \overline{(A_1)_n} & \hdots  & -2i \transp h_0 \overline{(A_d)_n} \\

\end{matrix}\right)$$
where we recall that  $(A_j)_l$ denotes, for $j=1,\dots,d$ and $l=1,\ldots,n$, the $l^{\rm th}$ column of $A_j$.
The last $d$ rows of  Equation (\ref{eqker2}) are of the form 
$$i\overline{\zeta} (1-\zeta)\tilde{g_j}(\zeta)-i\zeta \overline{(1-\zeta)\tilde{g_j}(\zeta)}=0,  \ j=1,\ldots,n,$$
 from which it follows that 
$$\tilde{g_j}=a_j-\overline{a_j}\zeta, \ a_j \in \C.$$
Solving the system backward, 
the previous $2n$ rows of Equation (\ref{eqker2}) are of the form
 \begin{equation*}
\overline{G_2(\zeta)}\left(\begin{matrix}
(1-\zeta)h  \\
(1-\zeta) \tilde{h}\\
\end{matrix}\right) 
+G_2(\zeta)\left(\begin{matrix}
\overline{(1-\zeta)h}  \\
 \overline{(1-\zeta)\tilde{h}}\\
\end{matrix}\right) + \overline{C(\zeta)} (1-\zeta)\tilde{g}+C(\zeta)\overline{(1-\zeta)\tilde{g}}=0.
\end{equation*}
Following the operations that lead to consider $G_2^{\flat}$ (see Equation (\ref{eqG_2})), we obtain 
 $n$ systems 
 \begin{equation*}
\overline{R(\zeta)}
\left(\begin{matrix}
(1-\zeta)h_k  \\
(1-\zeta) \tilde{h_k}\\
\end{matrix}\right) 
+R(\zeta)\left(\begin{matrix}
\overline{(1-\zeta)h_k}  \\
 \overline{(1-\zeta)\tilde{h_k}}\\
\end{matrix}\right) + \overline{C_k(\zeta)} (1-\zeta)\tilde{g}+C_k(\zeta)\overline{(1-\zeta)\tilde{g}} =0
\end{equation*}
with the abuse of notation $h=\transp \overline{A}h$ and where
$$C_k=\left(\begin{matrix}
2 \transp h_0 \overline{(A_1)_k} & \hdots  & 2 \transp h_0 \overline{(A_d)_k} \\
-2i \transp h_0 \overline{(A_1)_k} & \hdots  & -2i \transp h_0 \overline{(A_d)_k} \\
\end{matrix}\right).$$
Equivalently, we have  
\begin{eqnarray*}
\left(\begin{matrix}
(1-\zeta)h_k  \\
 (1-\zeta)\tilde{h_k}\\
\end{matrix}\right) 
&=&-\overline{R(\zeta)^{-1}}R(\zeta)\left(\begin{matrix}
\overline{(1-\zeta)h_k}  \\
 \overline{(1-\zeta)\tilde{h_k}}\\
\end{matrix}\right) -\overline{R(\zeta)^{-1}}\left(\overline{C_k(\zeta)} (1-\zeta)\tilde{g}+C_k(\zeta)\overline{(1-\zeta)\tilde{g}}\right)\\
\\
&=& -\left(\begin{matrix}
0 & \zeta       \\
\zeta & 0  \\
\end{matrix}\right)\left(\begin{matrix}
\overline{(1-\zeta)h_k}  \\
 \overline{(1-\zeta)\tilde{h_k}}\\
\end{matrix}\right) -\overline{R(\zeta)^{-1}}\left(\overline{C_k(\zeta)} (1-\zeta)\tilde{g}+C_k(\zeta)\overline{(1-\zeta)\tilde{g}}\right).
\end{eqnarray*}
Recall that  
\begin{equation*}
-\left(\begin{matrix}
0 & \zeta       \\
\zeta & 0  \\
\end{matrix}\right) = \Theta^{-1} \left(\begin{matrix}
\zeta & 0       \\
0 & \zeta  \\
\end{matrix}\right) \overline \Theta
\end{equation*}
with  
\begin{equation*}
\Theta= \left(\begin{matrix}
1 & -1       \\
i & i  \\
\end{matrix}\right).
\end{equation*}
 Therefore 
$$\Theta  \left(\begin{matrix}
(1-\zeta)h_k  \\
(1-\zeta)\tilde{h_k}\\
\end{matrix}\right) =\left(\begin{matrix}
\zeta & 0       \\
0 & \zeta  \\
\end{matrix}\right) \overline{\Theta \left(\begin{matrix}
(1-\zeta)h_k  \\
(1-\zeta)\tilde{h_k}\\
\end{matrix}\right) }-\Theta\overline{R(\zeta)^{-1}}\left(\overline{C_k(\zeta)} (1-\zeta)\tilde{g}+C_k(\zeta)\overline{(1-\zeta)\tilde{g}}\right).$$
Note that 
$$
\Theta\overline{R(\zeta)^{-1}}=\frac{1}{2}\left(\begin{matrix}
-(1-\zeta) & i(1+\zeta)       \\
i(1+\zeta) & 1-\zeta \\
\end{matrix}\right).
$$
Setting
$
\left(\begin{matrix}
u_k  \\
\tilde{u_k}\end{matrix}\right)
=
\Theta \left(\begin{matrix}
h_k  \\
\tilde{h_k}\\
\end{matrix}\right)$, we have
$$\left(\begin{matrix}
(1-\zeta)u_k  \\
(1-\zeta)\tilde{u_k}\end{matrix}\right)
 =\left(\begin{matrix}
\zeta & 0       \\
0 & \zeta  \\
\end{matrix}\right)\left(\begin{matrix}
 \overline{(1-\zeta)u_k}  \\
 \overline{(1-\zeta)\tilde{u_k}}\end{matrix}\right)
 -\Theta\overline{R(\zeta)^{-1}}\left(\overline{C_k(\zeta)} (1-\zeta)\tilde{g}+C_k(\zeta)\overline{(1-\zeta)\tilde{g}}\right).$$
and therefore, dividing by $1-\zeta$,

$$\left(\begin{matrix}
u_k  \\
\tilde{u_k}\end{matrix}\right)
 =-\left(\begin{matrix}
 \overline{u_k}  \\
 \overline{\tilde{u_k}}\end{matrix}\right)
 -\Theta\overline{R(\zeta)^{-1}}\left(\overline{C_k(\zeta)} \tilde{g}-C_k(\zeta)\overline{\zeta\tilde{g}}\right),$$
 that is 
$$\left(\begin{matrix}
u_k +\overline{u_k} \\
\tilde{u_k}+\overline{\tilde{u_k}}\end{matrix}\right)
 =
 -\Theta\overline{R(\zeta)^{-1}}\left(\overline{C_k(\zeta)} \tilde{g}-C_k(\zeta)\overline{\zeta\tilde{g}}\right).$$
A direct computation shows that 
$$\overline{C_k(\zeta)} \tilde{g}-C_k(\zeta)\overline{\zeta\tilde{g}}=\left(\begin{matrix}
v_k(\zeta)-\overline{\zeta v_k(\zeta)}  \\
i(v_k(\zeta)+\overline{\zeta v_k(\zeta)})\end{matrix}\right)$$
where 
$$v_k(\zeta)=2 \sum_{j=1}^d\transp \overline{h_0}(A_j)_k \tilde{g_{j}}=2 \sum_{j=1}^d\transp \overline{V}(A_j)_k (1-\overline{\zeta})(a_j-\overline{a_j}\zeta),$$
and 
$$-\Theta\overline{R(\zeta)^{-1}}\left(\overline{C_k(\zeta)} \tilde{g}-C_k(\zeta)\overline{\zeta\tilde{g}}\right)=
\left(\begin{matrix}
v_k(\zeta)+ \overline{v_k(\zeta)} \\
-i\left(v_k(\zeta)-\overline{v_k(\zeta)}\right)\end{matrix}\right).$$
Therefore 
\begin{equation*}
\begin{cases}
u_k +\overline{u_k} =v_k(\zeta)+ \overline{v_k(\zeta)}\\
\tilde{u_k}+\overline{\tilde{u_k}}=-i\left(v_k(\zeta)-\overline{v_k(\zeta)}\right)
\end{cases}
\end{equation*}
from which it follows that 
\begin{equation*}
\begin{cases}
u_k=  4 \sum_{j=1}^d\Re e (\transp \overline{V}(A_j)_k)\Re e (a_j)+iy_k -4\sum_{j=1}^d \Re e (\transp \overline{V}(A_j)_k)\overline{a_j}\zeta\\
\tilde{u_k}=4 \sum_{j=1}^d\Im m  (\transp \overline{V}(A_j)_k)\Re e (a_j) +i\tilde{y_k}-4\sum_{j=1}^d \Im m (\transp \overline{V}(A_j)_k)\overline{a_j}\zeta\\
\end{cases}
\end{equation*}
where $y_k, \tilde{y_k} \in \R$,  
and thus, since $h_k=1/2(u_k-i\tilde{u_k}),$ 
\begin{equation}\label{eqh}
h_k(\zeta)= 2 \sum_{j=1}^d\transp V\overline{(A_j)_k}\Re e (a_j)+\frac{\tilde{y_k}}{2}+i\frac{y_k}{2} -2\sum_{j=1}^d \transp V\overline{(A_j)_k}\overline{a_j}\zeta.
\end{equation}
We write $h=X+Y\zeta$.  
Now, the first $d$ rows of Equation (\ref{eqker2}) give rise to 
 \begin{eqnarray*}
\frac{1}{2}((1-\zeta)g+\overline{(1-\zeta)g})&=&- \overline{B(\zeta)}(1-\zeta)h- B(\zeta)\overline{(1-\zeta)h}
\end{eqnarray*}
and so 
 \begin{eqnarray*}
\frac{1}{2}(g-\overline{\zeta g})&=&- \overline{B(\zeta)}h+ B(\zeta)\overline{\zeta h}\\
&=&-(1-\overline{\zeta})(\overline{B_1}h+ B_1\overline{h}).
\end{eqnarray*}
Recall that $B$ and $B_1$ are defined in (\ref{eqB}); we use the inconsequential abuse of notation $B=(B,0)$ since $B$  just acts on the $h$ component of the discs.
This implies directly that  
\begin{equation}\label{eqg}
g(\zeta)= -4\Re e \left(\overline{B_1}X\right)+2\overline{B_1}Y-2\overline{B_1}Y\zeta.\\
\end{equation}
Now since $h(0)=0$ and $g(0)=0$ we have from (\ref{eqh}) and (\ref{eqg})
 \begin{eqnarray*}
\begin{cases}
X=0\\
\overline{B_1}Y=0.
\end{cases}
\end{eqnarray*}  
We will prove using the second set of equations, namely $\overline{B_1}Y=0$, that $a_1=\ldots=a_d=0$. We have 
\begin{eqnarray*}
\overline{B_1}Y&=&2\left(\begin{matrix}
\overline{(A_1)^1 V}& \hdots  &\overline{(A_1)^n V}  \\
\vdots &  & \vdots  \\

\overline{(A_d)^1 V}& \hdots  &\overline{(A_d)^n V}  \\

\end{matrix}\right)\left(\begin{matrix}
\sum_{j=1}^d \transp V\overline{(A_j)_1}\overline{a_j}  \\
\vdots  \\
\sum_{j=1}^d \transp V\overline{(A_j)_n}\overline{a_j}  \\
\end{matrix}\right)\\
& =&2 \underbrace{\left(\begin{matrix}
\overline{(A_1)^1 V}& \hdots  &\overline{(A_1)^n V}  \\
\vdots &  & \vdots  \\
\overline{(A_d)^1 V}& \hdots  &\overline{(A_d)^n V}  \\
\end{matrix}\right)}_{D_1}\underbrace{\left(\begin{matrix}
 \transp V\overline{(A_1)_1} & \hdots  & \transp V\overline{(A_d)_1} \\
\vdots &  & \vdots  \\
 \transp V\overline{(A_1)_n} & \hdots  & \transp V\overline{(A_d)_n}  \\
\end{matrix}\right)}_{D_2}
\left(\begin{matrix}
\overline{a_1}  \\
\vdots  \\

\overline{a_d} \\
\end{matrix}\right).\\
\\
\end{eqnarray*}
Recall that we have assumed $Q$ to satisfy conditon $\fff$ and thus $d\leq n$; in case $d> n$, $D_1$ and $D_2$ have rank less than $d-1$ and thus  the $d\times d$ matrix 
$D_1D_2$  cannot be invertible. In such case,  the linear map $d_{\bm f} \Psi_{\rho}: T_{\bm {f_0}}\mathcal{S}_0(Q) \to \C^{n+d}$ is not injective; note also that this will also be the case if $A_1,\ldots,A_d$ are not linearly independent.     
Note that  $D_1=\transp \overline{D_2}$. Since $V$ is given by condition $\fff$, it follows that $D_1$ is of rank $d$ and therefore 
$D_1D_2=\transp\overline{D_2}D_2$ is positive definite. It follows that $a_1=\ldots=a_d=0$. Since $X=0$, we also have $y_1=\ldots=y_d=\tilde{y_1}=\ldots=\tilde{y_d}=0$. This proves that $d_{\bm f} \Psi_{\rho}: T_{\bm {f_0}}\mathcal{S}_0(Q) \to \C^{n+d}$ is injective.
\end{proof}
\begin{remark}
Proposition \ref{propcenteropen} provides an open set $O \subset \C^{n+d}$ such that 
 $$O\subset \{f(0) \ \ | \ {\bm f} \in \mathcal{S}_0(\{r_t=0\}) , \ 
\|\bm{f}-\bm{f_0}\|_{\mathcal A^{1,\alpha}_{0}}<\varepsilon \}.$$
It follows directly from the proof that for any point $q \in O$, there exists an unique lift of stationary disc ${\bm f}$
such that $f(0)=q$.     
\end{remark}

\subsection{Injectivity of the jet map} 
Consider the linear jet map 
$$\mathfrak j_{1}: \left(\mathcal{A}^{1,\alpha}\right)^{2n+2d}  \to \mathbb C^{2(2n+2d)}$$ 
mapping ${\bm f}$ to its $1$-jet at $\zeta=1$, namely 
$$\mathfrak j_{1}({\bm f})=\left( {\bm f}(1), \displaystyle \frac{\partial {\bm f}}{\partial \zeta}(1)\right )\in \mathbb C^{2(2n+2d)}.$$
\begin{prop}\label{propjetdiscs}
Let $Q\subset \C^{n+d}$ be a quadric submanifold of real codimension $d$ given by (\ref{eqquadric}), fully non-degenerate at $0$.  Consider an initial disc $\bm{f_0}$ of the form 
 $$\bm{f_0}=\left((1-\zeta)V,2(1-\zeta)\transp \overline{V}A_1V,\ldots,2(1-\zeta)\transp \overline{V}A_dV,(1-\zeta)\transp \overline{V}A,\frac{\zeta}{2}c\right)$$ 
where $V$ is given by $\fff$ and $c=(c_1,\ldots,c_d)$ is chosen such that the matrix 
 $\sum_{j=1}^d c_jA_j$ is invertible. 
Then there exist an open
neighborhood $U$ of $\rho$ in $(\mathcal{C}^4(\B))^{d}$ and a positive $\varepsilon$ such that for all $r\in U$ the map  $\mathfrak j_{1}$ is injective on $\{\bm{f} \in \mathcal{S}_0(\{r=0\})\ \ | \ \|\bm{f}-\bm{f_0}\|_{1,\alpha}<\varepsilon\}$; in other words, such discs are determined by their $1$-jet at $1$.
\end{prop} 

Note that by the implicit function theorem, it is enough to prove that the restriction of $\mathfrak j_{1}$ to the tangent space $T_{\bm {f_0}}\mathcal{S}_0(Q)$ of $\mathcal{S}_0(Q)$ at the point $\bm {f_0}$ is injective. Moreover, the tangent space 
$T_{\bm {f_0}}\mathcal S_0(Q)$ is the kernel of the differential $\partial_2 F_0(\rho,\bm{f_0}^*)= 2\Re e \left[\overline{G(\zeta)}\ \cdot \right]$ where $F_0$, $\bm{f_0}^*$ and 
$G$ are respectively defined in (\ref{eqF_0}), (\ref{discini}) and (\ref{eqG}). Proposition \ref{propjetdiscs} is then a consequence of the following lemma. 
\begin{lemma}
The restriction of $\mathfrak j_{1}$ to the kernel of $2\Re e \left[\overline{G(\zeta)}\  \cdot \right]$ is injective.
\end{lemma}
\begin{proof}
We  follow strictly the notations introduced in the proofs of Theorem \ref{theodiscs} and Lemma \ref{leminj}. The main step of the proof is to describe the kernel of  
$2\Re e \left[\overline{G(\zeta)}\ \cdot \right]$; this was exactly done in Lemma \ref{leminj}. Let $\bm{f}=(1-\zeta)(h,g,\tilde{h},\tilde{g})$ be an element of the kernel of $2\Re e \left[\overline{G(\zeta)}\ \cdot \right]$ with trivial $1$-jet at $1$. We have 
\begin{eqnarray*}
\begin{cases}
\tilde{g_j}=a_j-\overline{a_j}\zeta, \ a_j \in \C\\
h=X+Y\zeta\\
g(\zeta)= -4\Re e \left(\overline{B_1}X\right)+2\overline{B_1}Y-2\overline{B_1}Y\zeta.\\
\end{cases}
\end{eqnarray*}  
Since $\bm{f}$ has a trivial $1$-jet at $1$, we must have 
\begin{eqnarray*}
\begin{cases}
a_j \in \R\\
X=-Y\\
\Re e \left(\overline{B_1}X\right)=0.\\
\end{cases}
\end{eqnarray*}  
It follows that 
$$\overline{B_1}X=-\overline{B_1}Y=-\transp\overline{D_2}D_2\left(\begin{matrix}
a_1  \\
\vdots  \\
a_d \\
\end{matrix}\right) \in \R$$
and since $\transp\overline{D_2}D_2$ is positive definite, we must have $a_1=\ldots=a_d=0$. This implies directly that $g=\tilde{g}=0$ and $h=0$. Equation (\ref{eqh}) implies that $y_k=\tilde{y_k}=0$ and thus  $u_k=\tilde{u_k}=0$ which finally leads to $\tilde{h}=0$.  
\end{proof}

\section{Jet determination of CR automorphisms} 
Let $k$ be a positive integer.   
Let $M\subset \C^N$ be a $\mathcal{C}^4$ generic  real submanifold and let $p \in M$. We denote by $Aut^k(M,p)$ the set of germs at $p$  of CR automorphisms $F$ of $M$ of class $\mathcal{C}^k$; in particular we have $F(p)=p$ and $F(M)\subset M$.    
\begin{theo}\label{theojet}
Let $M\subset \C^N$ be a $\mathcal{C}^4$ generic  real submanifold. Assume that $M$ is   fully non-degenerate at  $p \in M$. Then elements of $Aut^3(M,p)$ are uniquely determined by their $2$-jet at $p$.     
\end{theo}
In order to prove Theorem \ref{theojet}, we need a technical lemma. 
Let $M\subset \C^N$ be a $\mathcal{C}^4$ generic  real submanifold  given by 
 $\{r=0\}$ (\ref{eqred}), and let $Q$ be its associated quadric part  defined by  $\{\rho=0\}$ (\ref{eqquadric}).  We recall the following anisotropic dilation $\Lambda_t: \C^{n+d} \to \C^{n+d}$ given by 
$$\Lambda_t(z,w)=(tz,t^2w).$$ We set $M_t=\Lambda_t^{-1}(M)$, $r_t=\frac{1}{t^2}r\circ \Lambda_t$ and $F_t=\Lambda_t^{-1}\circ F \circ \Lambda_t$, where $F \in Aut^3(M,p)$. We also recall that for an analytic disc ${\bm f}=(f,\tilde{f}) \in  \left(\mathcal{A}^{1,\alpha}\right)^{2n+2d}$ where $0<\alpha<1$, we have  
$$(F_t)_*{\bm f}(\zeta)=\left(F_t\circ f(\zeta), \tilde{f}(\zeta)(d_{f(\zeta)}F_t)^{-1}\right)$$
for $\zeta \in \Delta$,  and $F_t\circ f(\zeta)$ is well defined thanks to Proposition 6.2.2 \cite{ber}.   
The following lemma follows from  the same arguments used in  Section 5.1 \cite{be-bl} or Section 5.1 \cite{be-de1}. 
\begin{lemma}\label{lemdil}
We have the following:
\begin{enumerate}[i.] 
\item Let $U$ be a neighborhood of $\rho$ in $\mathcal{C}^4$ topology. Then for $t$ small enough, $r_t \in U$.      
\item Let $F \in Aut^3(M,p)$ with a trivial $2$-jet. There exists a positive constant $K$ such that for all ${\bm f}=((1-\zeta)h,(1-\zeta)g,(1-\zeta)\tilde{h},(1-\zeta)\tilde{g}+c\zeta) 
 \in  \mathcal{A}$ we have 
$(F_t)_*{\bm f} \in \mathcal{A}$ and 
$$\|(F_t)_*{\bm f}-\bm{f}\|_{\mathcal A^{1,\alpha}_{0}}<tK \left(\max \{\|h\|_{1,\alpha},\|g\|_{1,\alpha},\|\tilde{h}\|_{1,\alpha},\|\tilde{g}\|_{1,\alpha}\}\right)^3.$$
\end{enumerate}
\end{lemma}   

\begin{proof}[Proof of Theorem \ref{theojet}]
Let $M\subset \C^N$ be a $\mathcal{C}^4$ generic  real submanifold, fully non-degenerate at at $p\in M$. We may assume that $p=0$ and that $M$ is locally given by $\{r=0\}$ (\ref{eqred}).
Denote by $Q$ the associated quadric part of $M$ defined by $\{\rho=0\}$ (\ref{eqquadric}). 
Let $F \in Aut^3(M,0)$ with a trivial $2$-jet at $0$. We wish to show that $F$ is the identity.

Since $M$ satisfies $\fff$ at $0$, there exists $V\in \C^{n}$ such that ${\rm span}_{\C}\{A_1V,\ldots,A_dV\}$ is of dimension $d$. Consider an initial lift of stationary disc $\bm{f_0}$ of the form 
 $$\bm{f_0}=\left((1-\zeta)V,2(1-\zeta)\transp \overline{V}A_1V,\ldots,2(1-\zeta)\transp \overline{V}A_dV,(1-\zeta)\transp \overline{V}A,\frac{\zeta}{2}c\right)$$
 where  $c_1,\ldots,c_d$ are chosen such that $\sum_{j=1}^d c_jA_j$ is invertible.  Denote by $U$ the  neighborhood of $\rho$ in $\mathcal{C}^4$ topology obtained in  Theorem \ref{theodiscscons}. According to Lemma \ref{lemdil}, for $t$ small enough, the defining functions $\displaystyle r_t=\frac{1}{t^2}r\circ
 \Lambda_t \in U$. Reducing $U$, and therefore $t$, if necessary, Proposition \ref{propcenteropen} provides an open set $O \subset \C^{n+d}$ such that 
 $$O\subset \{f(0) \ \ | \ {\bm f} \in \mathcal{S}_0(\{r_t=0\}) , \ 
\|\bm{f}-\bm{f_0}\|_{\mathcal A^{1,\alpha}_{0}}<\varepsilon/2 \}.$$We will show that $F_t=\Lambda_t^{-1}\circ F \circ \Lambda$ is equal to the identity on the open set $O$. Let $q \in O$ and let ${\bm f}$ be the (unique) lift of stationary disc in  $\mathcal{S}_0(\{r_t=0\})$ with $\|\bm{f}-\bm{f_0}\|_{\mathcal A^{1,\alpha}_{0}}<\varepsilon/2$ and such that $f(0)=q$. By invariance and since $F_t$ has a trivial $2$-jet, we have  $(F_t)_*{\bm f} \in  \mathcal{S}_0(\{r_t=0\})$ and by Lemma \ref{lemdil} we have $\|(F_t)_*{\bm f}-\bm{f_0}\|_{\mathcal A^{1,\alpha}_{0}}<\varepsilon$ for $t$ small enough. Moreover, the discs $(F_t)_*{\bm f}$ and ${\bm f}$ have the same $1$-jet. By Proposition \ref{propjetdiscs} we have $(F_t)_*{\bm f}={\bm f}$ and therefore $F_t \circ f(0)=f(0)$, that is $F_t (q)=q$. This achieves the proof of Theorem \ref{theojet}.  
\end{proof}

\ \ \

\noindent  {\it Acknowledgments.}  This work was done when authors were visiting the Center for Advanced Mathematical Sciences of the American University of Beirut, the Department of the Mathematics of the University of Fribourg and of the University of Lille. We are thankful for the hospitality and the support of these institutions.

\vskip 1cm
{\small
\noindent Florian Bertrand\\
Department of Mathematics, Fellow at the  Center for Advanced Mathematical Sciences (CAMS)\\
American University of Beirut, Beirut, Lebanon\\
{\sl E-mail address}: fb31@aub.edu.lb\\

\noindent L\'ea Blanc-Centi \\
Laboratoire Paul Painlev\'e\\
Universit\'e de Lille, 59655 Villeneuve d'Ascq C\'edex, France\\
{\sl E-mail address}:  lea.blanc-centi@univ-lille.fr\\

\noindent Francine Meylan \\
Department of Mathematics\\
University of Fribourg, CH 1700 Perolles, Fribourg\\
{\sl E-mail address}: francine.meylan@unifr.ch\\
} 

\end{document}